\documentclass[11pt,reqno]{amsart}

\usepackage{amssymb}

\usepackage[left=2.8cm,right=2.8cm,top=2.9cm,bottom=2.9cm]{geometry}

\usepackage{amsfonts}
\usepackage{color}
\usepackage{graphicx}
\usepackage[bookmarksnumbered,colorlinks, linkcolor=blue, citecolor=red, pagebackref, bookmarks, breaklinks]{hyperref}

\newtheorem{thm}{Theorem}
\newtheorem{cor}[thm]{Corollary}
\newtheorem{lema}[thm]{Lemma}
\newtheorem{prop}[thm]{Proposition}
\theoremstyle{definition}

\theoremstyle{remark}

\newtheorem{rem}[thm]{Remark}
\newcommand{\R}{\mathbb R}
\newcommand{\N}{\mathbb N}

\newcommand{\vbar}[1]{{\left\vert\kern-0.25ex\left\vert\kern-0.25ex\left\vert #1 
    \right\vert\kern-0.25ex\right\vert\kern-0.25ex\right\vert}}

\def\C{\mathbf {C}}

\newcommand{\ve}{\varepsilon}
\newcommand{\lam}{\lambda}

\title[Global H\"older regularity for eigenfunctions of the fractional $g$-Laplacian]{Global H\"older regularity for eigenfunctions of the fractional $g$-Laplacian}

\author{Juli\'an Fern\'andez Bonder, Ariel Salort and Hern\'an Vivas}

\address[JFB and AS]{Instituto de C\'alculo (IC), CONICET\\
Departamento de Matem\'atica, FCEN - Universidad de Buenos Aires\\
Ciudad Universitaria, Pabell\'on I, C1428EGA, Av. Cantilo s/n\\
Buenos Aires, Argentina}
	
\email[JFB]{jfbonder@dm.uba.ar}

\email[AS]{asalort@dm.uba.ar}

\address[HV]{Centro Marplatense de Investigaciones Matem\'aticas (UNMDP-CIC) \\
Departamento de Matem\'atica, FCEYN, UNMDP, De\'an Funes 3350, 7600, Mar del Plata, Argentina}

\email{havivas@mdp.edu.ar}



\begin{document}
 
\begin{abstract}
We establish global H\"older regularity for eigenfunctions of the fractional $g-$Laplacian with Dirichlet boundary conditions where $g=G'$ and $G$ is a Young functions satisfying the so called $\Delta_2$ condition. Our results apply to more general semilinear equations of the form $(-\Delta_g)^s u  = f(u)$.  
\end{abstract}

\maketitle

\section{Introduction}	 

The aim of this note is to prove global H\"older regularity for solutions of the non-linear eigenvalue problem
\begin{equation}\label{eq.main}
\begin{cases}
(-\Delta_g)^s u  =\lam g(u) &\quad \text{ in } \Omega,\\
u = 0 &\quad \text{ in } \R^n \setminus \Omega
\end{cases}
\end{equation} 
where $\Omega\subset \R^n$ is an open and bounded set with Lipschitz boundary, $\lambda\in\R$, $g$ is the derivative of a Young function $G$ (see Section \ref{sec.preliminares} for definitions), $(-\Delta_g)^s$ the fractional $g-$Laplacian is given by
\[
(-\Delta_g)^su(x)=\text{p.v.}\int g\left(D_su(x,y)\right)\frac{dy}{|x-y|^{n+s}}
\]
and $D_su$ is the $s-$H\"older quotient defined by
\begin{equation}\label{eq.squotient}
D_su(x,y):=\frac{u(x)-u(y)}{|x-y|^s}.
\end{equation}

The fractional $g-$Laplacian operator is the natural generalization of the fractional $p-$Laplacian when a non-power behavior of the $s-$H\"older quotient is considered. These type of operators have received much attention in  recent years, see for instance \cite{FBS, BO, S1, SV, SCAB, BS,  FBPLS, FBSV-homog, JC, S} and the references in these articles. Observe that in the particular case that $G(t)=t^p$, $p>1$, the eigenvalue problem for the fractional $p-$Laplacian is recovered. 

The non-local, non-linear and non-homogeneous eigenvalue problem \eqref{eq.main} was treated in \cite{S, SV}, where existence of eigenvalues was proved.  Eigenvalues with other boundary conditions were studied in \cite{BS}. Recently, a homogeneous version of \eqref{eq.main} was dealt with in \cite{FBSV-homog}. The particular case of powers was studied in \cite{BP, FP,LL}, where $L^\infty$ bound of eigenfunctions was obtained.

The local version of \eqref{eq.main} was addressed for instance in \cite{autov1, autov2, ML}. In \cite{ML}, by appealing to the regularity 	theory of G. Lieberman, the authors  prove $C^{1,\alpha}$ regularity of the first eigenfunction. In this setting, lower bounds of eigenvalues were also proved in \cite{S1}.
  
Our goal in this manuscript is to study regularity of solutions of \eqref{eq.main} for a class of operators where the Young function $G$ satisfies that $G(t)$  is comparable with $tG'(t)$ (see condition \eqref{cond}) and %$G$ is sub-multiplicative (see condition \eqref{sub}) and 
$G$ is sub-critical in the sense of condition \eqref{condi}. This class of Young functions includes powers, powers multiplied by logarithms and sum of different powers, among others functions (see Section \ref{sec.preliminares} for further examples).

The desired regularity result will be achieved via the following a priori bound on the $L^\infty$ norm of $u$, which is the main result of this paper:

\begin{thm} \label{main}
Let $\mu>0$, $s\in(0,1)$, $1<p^-<p^+<\infty$ and $G$ a Young function satisfying 
\[
p^-\leq \frac{tg(t)}{G(t)} \leq p^+
\]
with $sp^+<n$ and $p^+\le \frac{np^-}{n-sp^-}$ (the Sobolev conjugate of $p^-$).

There exists a constant $C_0=C_0(n,s,p^\pm,\mu)>0$ such that if $u\in W^{s,G}_0(\Omega)$ is a weak solution of \eqref{eq.main} with $\int_\Omega G(|u|)\,dx=\mu$ then $u$ is bounded in $\Omega$ with
$$
\|u\|_{L^\infty(\Omega)} \leq C_0.
$$
\end{thm}

\begin{rem}
We point out that the restriction $p^+\le \frac{np^-}{n-sp^-}$ is only used in the proof of Lemma \ref{lem.H} and is of technical character. An improvement in the proof of that Lemma would imply the elimination of that restriction in the regularity result.
\end{rem}

Global regularity is a consequence of the previous theorem and the regularity result recently proved by the authors \cite[Theorem 1.1]{bonder2020interior}; it basically states that solutions of an equation with bounded right hand side are globally H\"older continuous. We point out that these result hold for values of $p^->2$ and a convex function $g$; in the case of powers, that is when $G(t)=t^p$, these conditions mean that we are dealing with the degenerate (or convex) scenario. The precise statement is as follows:

\begin{thm}\label{thm.reg}
Let $s\in(0,1)$, $2<p^-<p^+<\infty$ and $G$ a Young such that $g=G'$ is convex and satisfies
\begin{equation} \label{cond.lib}
 p^--1\leq\frac{tg'(t)}{g(t)}\leq p^+-1.
\end{equation}
Assume that $\Omega$ is an open bounded set with $C^{1,1}$ boundary. Finally, let $h\in L^\infty(\Omega)$.

There exist constants $\alpha=\alpha(s,n,p^\pm,\|h\|_{L^\infty(\Omega)})\in(0,s]$ and $C_1=C_1(s,n,p^\pm,\|h\|_{L^\infty(\Omega)},[\partial\Omega]_{C^{1,1}})>0$ such that if $u\in W^{s,G}_0(\Omega)$ is a weak solution of 
\[
\left\{ \begin{array}{cccl}
(-\Delta_g)^s  u  & = & h &  \textrm{ in }\Omega \\
u & = & 0  &\textrm{ in }\R^n\setminus\Omega,
\end{array} \right.
\]
then $u\in C^\alpha(\overline\Omega)$ and 
\[
\| u \|_{C^\alpha (\overline\Omega)}\leq  C_1.
\]
\end{thm}

Indeed, since Theorem \ref{main} states that eigenfunctions are bounded, we immediately have the following corollary:
\begin{cor}\label{cor}
Assume that the hypothesis of Theorem \ref{main} hold, $p^->2$, $g=G'$ is convex and satisfies \eqref{cond.lib} and $\Omega$ has $C^{1,1}$ boundary. 

There exists $\alpha=\alpha(s,n,p^\pm,\lambda,\mu)\in(0,s]$ and $C_3=C_3(s,n,p^\pm,\lambda,\mu,[\partial\Omega]_{C^{1,1}})>0$  such that if $u$ is a weak solution of \eqref{eq.main} then $u\in C^\alpha(\overline\Omega)$ and 
\[
\| u \|_{C^\alpha (\overline\Omega)}\leq  C_3.
\]
\end{cor}

The proof of Theorem \ref{main} is based on De Giorgi's  iteration scheme. This technique is a powerful tool in regularity analysis of elliptic and parabolic PDEs and it has been shown to be very versatile and adaptable to different contexts. 

We would like to highlight that our setting presents specific challenges when trying to apply the aforementioned technique. One of the main drawbacks to be overtaken in order to apply the iteration scheme is the possible lack of homogeneity of the equation; this induces rather important technical difficulties as nor the uniformly elliptic argument nor the degenerate ($p-$Laplacian) one apply mutatis mutandis. An extra, and related, difficulty is added by the fact that weak solutions satisfy an equation in terms of modular but not in terms of norms, however, both embedding theorems and H\"older type inequalities hold for norm of functions and not for modulars.

\medskip
	
A slight modification of our arguments gives regularity of nonlinear problems of the form 
\begin{equation}\label{eq.main.f}
\begin{cases}
(-\Delta_g)^s u  = f(u) &\quad \text{ in } \Omega,\\
u = 0 &\quad \text{ in } \R^n \setminus \Omega
\end{cases}
\end{equation} 
where the nonlinearity satisfies $f=F'$ for $F$ a Young function such that
\[
\eta^-\leq\frac{tf(t)}{F(t)}\leq \eta^+.
\]
together with the sub-critical constrain $F\prec\prec G^*$, that is  
\begin{equation}\label{eq.subcrit}
\lim_{t\to\infty} \frac{F(kt)}{G^*(t)}=0;
\end{equation}
this condition is enough to ensure that the embeddings hold. Here $G^*$ is the Sobolev Young function defined in \eqref{G*1}. Condition \eqref{eq.subcrit} holds in particular if $\eta^+ \leq (p^-)^*$, the Sobolev conjugate of $p^-$ which is the lower bound 
\[
(p^-)^*\leq \frac{t(G^*)'(t)}{G^*(t)}
\]
see \cite{Fukagai}.

\begin{thm}\label{thm.semilinear}
Assume that the hypothesis of Theorem \ref{main} hold and that $F$ is a Young function satisfying \eqref{eq.subcrit}.

There exists $C_4=C(n,s,p^\pm,\eta^\pm)>0$ such that if $u\in W^{s,G}_0(\Omega)$ is a weak solution of \eqref{eq.main.f} then $u$ is bounded and 
$$
\|u\|_{L^\infty(\Omega)} \leq C_4.
$$

If additionally the hypothesis of Theorem \ref{thm.reg} are satisfied, there exist $\alpha=\alpha(n,s,p^\pm,\eta^\pm)\in (0,s]$ and $C_5=C_5(n,s,p^\pm,\eta^\pm,[\partial\Omega]_{C^{1,1}})>0$ such that $u\in C^\alpha(\overline\Omega)$ and 
$$
\|u\|_{C^\alpha(\overline\Omega)} \leq C_5.
$$
\end{thm}

We would like to highlight the recent work \cite{CSAB} in which regularity estimates for quasilinear equations driven by the $g-$Laplacian are addressed via a Moser type approach. These results bear some resemblance with the ones presented here. However, we point out that they assume a stronger assumptions on the Young function $G$, i.e., the submultiplicativity condition (condition known as the $\Delta'$ condition). Moreover, in \cite{CSAB} the eigenvalue problem is only covered  for the case in which $g$ is equivalent to a power. We also allow for a broader range of growth behaviors in the semilinear setting \eqref{eq.main.f}.

The paper is organized as follows: in Section \ref{sec.preliminares} we give the necessary definitions and provide some examples of Young functions that fit our setting; in Section \ref{sec.tech} we prove some technical results that will be used in the proof of Theorem \ref{main}; in Section \ref{sec.main} we prove Theorem \ref{main} (and as a consequence of Theorem \ref{thm.reg} we get Colorllary \ref{cor}).

\section{Preliminaries} \label{sec.preliminares}

An application $G\colon[0,\infty)\longrightarrow [0,\infty)$ is said to be a  \emph{Young function} if it admits the integral representation 
\[
G(t)=\int_0^t g(s)\,ds,
\] 
where the right-continuous function $g$ defined on $[0,\infty)$ has the following properties:
\begin{itemize}
\item[(i)] $g(0)=0$, \quad $g(t)>0 \text{ for } t>0$, 
\item[(ii)]  $g \text{ is nondecreasing on } (0,\infty)$, 
\item[(iii)]  $\lim_{t\to\infty}g(t)=\infty$.
\end{itemize}
From these properties it is easy to see that a Young function $G$ is continuous, nonnegative, strictly increasing and convex on $[0,\infty)$.  Without of loss generality we can assume $G(1)=1$ and we extend $G$ to negative values in an even fashion: $G(-t)=G(t)$. 

We will assume throughout the paper that $1<p^-<p^+<\infty$ and that $G$ satisfies 
\begin{equation} \label{cond} 
p^-\leq \frac{tg(t)}{G(t)} \leq p^+.
\end{equation}
%which, integrating the inequalities, amounts to
%\begin{equation}\label{cond2}
%t^{p^-} \leq G(t) \leq t^{p^+}.
%\end{equation}
Condition \eqref{cond} is equivalent to ask $G$ and $\tilde G$ to satisfy the \emph{$\Delta_2$ condition or doubling condition}, i.e., 
\[
G(2t)\leq 2^{p^+} G(t), \qquad \tilde G(2t) \leq 2^{(p^-)'} \tilde G(t),
\]
(we usually denote $\C:=2^{p^+}$) where the \emph{complementary} function of a Young function $G$ is the Young function $\tilde G$ defined as
$$
\tilde G(t)=\sup\{ta-G(a)\colon a>0\}.
$$
This condition allows to split sums as
\begin{equation}\label{eq.delta2'}
G(a+b)\leq \tfrac{\C}{2} (G(a)+G(b)).
\end{equation}

The following lemma will be useful often; its proof is elementary so we omit it.
\begin{lema}
For $\alpha\in[0,1]$ and  $t\geq 0$
$$
G(\alpha t) \leq \alpha G(t), 
$$
and for $\alpha\geq 1$ and  $t\geq 0$
$$
G(\alpha t)\geq \alpha G(t). 
$$
More generally,  for any, $\alpha,t\geq 0$
\begin{equation}\label{minmax2} 
G(t)\min\{\alpha^{p^-}, \alpha^{p^+}\} \leq G(\alpha t) \leq G(t)\max\{\alpha^{p^-}, \alpha^{p^+}\},
\end{equation}
\begin{equation}\label{minmax3} 
G^{-1}(t)\min\{\alpha^\frac{1}{p^-}, \alpha^\frac{1}{p^+}\} \leq G^{-1}(\alpha t) \leq G^{-1}(t)\max\{\alpha^\frac{1}{p^-}, \alpha^\frac{1}{p^+}\}.
\end{equation}
\end{lema}

\medskip

Examples of Young functions satisfying the assumptions of Theorem \ref{main} include:
\begin{itemize}
\item $G(t)=t^p$, $t\geq0$, $p>1$;
\item $G(t)=t^p (1+|\log t |)$, $t\geq0$, $p>1$;
\item $G(t)= t^p \chi_{(0,1]}(t) + t^q \chi_{(1,\infty)}(t)$, $t\geq0$, $p,q>1$;
\item $G(t) = t^p + t^q$, $t\geq0$, $p,q>1$;
\item $G(t)$ given by the complementary function to $\tilde G(t)=(1+t)^{\sqrt{\log(1+t)}}-1$, $t\geq0$;
\item $G_1\circ \ldots\circ G_m$, $\max\{G_1,\ldots,G_m\}$ and $\sum_{j=1}^ma_jG_j$ where $G_j$ is a Young function and $a_j\geq 0$ for $j=1,\ldots,m$.
\end{itemize}

We will assume also that 
\begin{equation} \label{condi}
\displaystyle\int_{0}^{1}\frac{G^{-1}(\tau)}{\tau^{1+\frac{s}{n}}}\:d\tau<\infty\ \quad \text{and}\quad
\displaystyle\int_{1}^{+\infty}\frac{G^{-1}(\tau)}{\tau^{1+\frac{s}{n}}}\:d\tau=\infty
\end{equation}
which are the conditions necessary for the Orlicz-Sobolev embeddings to hold (see Proposition \ref{embedding}). We also consider the Young function $G^*$ defined as
\begin{equation} \label{G*1}
(G^*)^{-1}(t):=\int_0^t \frac{G^{-1}(\tau)}{\tau^\frac{n+s}{n}}\,d\tau.
\end{equation}
Condition \eqref{condi} is fulfilled in particular when
\[
sp^+<n.
\]

From now on, the modulars in $L^G(\Omega)$ and $W^{s,G}(\Omega)$ will be denoted as
$$
\Phi_G(u):=\int_\Omega G(|u|)\,dx \qquad \Phi_{s,G}(u):=\iint_{\R^n\times \R^n} G(|D_s u|)\,d\mu,
$$
respectively where the notation 
\[
d\mu:=\frac{dxdy}{|x-y|^n}
\]
and $D_s u$ defined in \eqref{eq.squotient} will be used throughout. Over the space 
$$
W^{s,G}(\Omega):=\{u\colon \R^n\to \R \text{ measurable s.t. } \Phi_{s,G}(u)+\Phi_G(u)<\infty\}
$$ 
we define the norm 
\[
\|u\|_{s,G} := \|u\|_G + [u]_{s,G},
\]
where
\[
\|u\|_G :=\inf\left\{\lambda>0\colon \Phi_G\left(\frac{u}{\lambda}\right)\le 1\right\}
\]
and
\[
[u]_{s,G} :=\inf\left\{\lambda>0\colon \Phi_{s,G}\left(\frac{u}{\lambda}\right)\le 1\right\}.
\]
is the {\em $(s,G)-$Gagliardo seminorm}. 

We also denote 
$$
W^{s,G}_0(\Omega):=\{u\in W^{s,G}(\Omega)\colon u=0\text{ a.e. in } \R^n\setminus \Omega\}.
$$
In this space $[\cdot]_{s,G}$ turns out to be an equivalent norm. 

For the proof of the aforementioned facts and an introduction to fractional Orlicz-Sobolev spaces we refer to \cite{FBS, FBPLS}. The following embedding is proved in \cite{BO}.  

\begin{prop}[Embedding] \label{embedding}
Let $\Omega\subset \R^n$ be a Lipschitz domain and let $G$ be a Young function satisfying \eqref{cond} and \eqref{condi}, then there is a positive constant $C$ such that
$$
\|u\|_{G^*} \leq C \|u\|_{s,G}.
$$
\end{prop}

We point out that \eqref{condi}-\eqref{G*1} is not the most general setting under which embedding results for $W^{s,G}(\Omega)$ are known to hold; indeed, in \cite{ACLS} embeddings are proved for 
\[
G_{\frac{s}{n}}(t):=G\left(B^{-1}(t)\right)\quad\text{with}\quad B(t):=\left(\int_0^t\left(\frac{\tau}{G(\tau)}\right)^{\frac{s}{s-n}}\:d\tau\right)^{\frac{n-s}{n}}.
\]
and proven to be sharp; consequently $G_{\frac{s}{n}}$ is the critical (optimal) Sobolev conjugate of $G$. However, the simplicity of the formula \eqref{G*1} for $G^*$ allows us to simplify a lot of our arguments.

Throughout the paper, given a Young function $G$ we will denote 
\begin{equation}\label{eq.h}
H:=G^*\circ G^{-1},
\end{equation}
where $G^*$ is the Sobolev Young function defined in \eqref{G*1} and $G^{-1}$ is the inverse of $G$. 
Observe that $H$ defines a new Young function.

\begin{lema}\label{lem.H}
Assume that $g$ satisfies \eqref{cond.lib} and, moreover, that $p^+\le (p^-)^*$. Then, the function $H$ defined in \eqref{eq.h} is a Young function.
\end{lema}

\begin{proof}
The proof is done by brute force computations, so we only give a hint of the computations that need to be done.

We will be using a fact on the Sobolev conjugate function $G^*$ that is proved in \cite{Fukagai} (see also \cite{BO})
\begin{equation}\label{G*delta}
(p^-)^*\le \frac{(G^*)'(t)t}{G^*(t)}\le (p^+)^*
\end{equation}
Now, in order to show that $H$ is a Young function, the only nontrivial issue is to show that $H$ is convex. But this turns out to be equivalent to prove that $(H^{-1})''(t)\le 0$.

By direct computations, we arrive at
$$
(H^{-1})''(t) = g((G^*)^{-1}(t))\left(\frac{g'((G^*)^{-1}(t))}{g((G^*)^{-1}(t))}\left[\frac{G^{-1}(t)}{t^{1+s/n}}\right]^2 + \frac{1}{g(G^{-1}(t)) t^{1+s/n}} - \frac{(1+s/n) G^{-1}(t)}{t^{2+s/n}}\right).
$$

So $(H^{-1})''(t)\le 0$ is equivalent to
$$
\frac{g'((G^*)^{-1}(t))}{g((G^*)^{-1}(t))}\left[\frac{G^{-1}(t)}{t^{1+s/n}}\right]^2 + \frac{1}{g(G^{-1}(t)) t^{1+s/n}}\le \frac{(1+s/n) G^{-1}(t)}{t^{2+s/n}},
$$
which is also equivalent to
\begin{equation}\label{cond.H2da}
\frac{g'((G^*)^{-1}(t))}{g((G^*)^{-1}(t))} \frac{G^{-1}(t)}{t^{1+s/n}} + \frac{1}{g(G^{-1}(t)) G^{-1}(t)}\le \frac{(1+s/n)}{t}.
\end{equation}
Now, using \eqref{cond.lib} and \eqref{cond}, we have that
$$
\frac{g'((G^*)^{-1}(t))}{g((G^*)^{-1}(t))} \frac{G^{-1}(t)}{t^{1+s/n}} + \frac{1}{g(G^{-1}(t)) G^{-1}(t)}\le (p^+-1)\frac{G^{-1}(t)}{(G^*)^{-1}(t) t^{1+s/n}} + \frac{1}{p^-}\frac{1}{t}.
$$
So, \eqref{cond.H2da} will hold true if we have that
$$
(p^+-1)\frac{G^{-1}(t)}{(G^*)^{-1}(t) t^{1+s/n}} + \frac{1}{p^-}\frac{1}{t}\le \frac{(1+s/n)}{t},
$$
which in turn is equivalent to
$$
\frac{t[(G^*)^{-1}]'(t)}{(G^*)^{-1}(t)}\le \frac{(p^-)^*-1}{(p^+-1)(p^-)^*}.
$$
Finally, \eqref{G*delta} implies that
$$
\frac{t[(G^*)^{-1}]'(t)}{(G^*)^{-1}(t)}\le \frac{1}{(p^-)^*},
$$
so to finish the proof we just have to observe that $p^+\le (p^-)^*$ implies that
$$
\frac{1}{(p^-)^*}\le \frac{(p^-)^*-1}{(p^+-1)(p^-)^*}
$$
and so the proof is complete.
\end{proof}

\medskip

Weak solutions of \eqref{eq.main} satisfy
\begin{equation}\label{eq.weak}
\langle (-\Delta_g)^s u, v\rangle = \lambda \int_\Omega g(u)v\,dx\quad\text{ for any }v\in W^{s,G}_0(\Omega)
\end{equation}
where
\[
\langle (-\Delta_g)^s u, v\rangle :=\iint_{\R^n\times\R^n} g\left(\frac{u(x)-u(y)}{|x-y|^s}\right)\frac{(v(x)-v(y))}{|x-y|^{n+s}}\:dxdy.
\]
%If we recall \eqref{eq.squotient} and \eqref{eq.mu}, \eqref{eq.weak} equality can be written in the more compact form
%\[
%\iint g\left(D_su(x,y)\right)D_sv(x,y)\:d\mu=\int_\Omega f(u)v\,dx.
%\]

\section{Some technical results}\label{sec.tech} 

The purpose of this section is to gather some technical and useful inequalities which are the key of our argument. Recall \eqref{eq.h} for the definition of $H$. We start with the following lemma that is needed to bound the $H$ norm of the characteristic function of upper level sets of truncations:  
 
\begin{lema}\label{cotaK}
Let $G$ be a Young function satisfying \eqref{cond} and \eqref{condi} and $G^*$ be defined by \eqref{G*1}. If we define
\begin{align*}
K(t):=t(G^*\circ G^{-1})^{-1}\left(\tfrac1t\right)= t (G\circ (G^*)^{-1})\left(\tfrac1t\right)
\end{align*}
then there exists a some constant $C=C(n,s,p^\pm,\C)>0$ such that
\[
K(t)\leq C \max\{ t,  t^\frac{sq}{n} \} \quad\text{for all }t>0 \text{ and } q<p^-.
\]
\end{lema}

\begin{proof}

Using the expression of $G^*$ and \eqref{condi} we have that, for $t\geq 1$, 
$$
K(t)= tG\left( \int_0^{\frac1t} \frac{G^{-1}(\tau)}{\tau^{1+\frac{s}{n}}}\:d\tau \right) \leq tC_1 \qquad \text{ with } C_1:=G\left(\int_0^1 \frac{G^{-1}(\tau)}{\tau^{1+\frac{s}{n}}}\,d\tau\right).
$$
Let us deal now with the case $t<1$. From \eqref{eq.delta2'} we get
\begin{align*}
K(t)&= tG\left(\int_0^1 \frac{G^{-1}(\tau)}{\tau^{1+\frac{s}{n}}}\,d\tau+ \int_1^{\frac1t} \frac{G^{-1}(\tau)}{\tau^{1+\frac{s}{n}}}\,d\tau\right)  \\
	& \leq \frac{\C}{2}t\left(G\left(\int_0^1 \frac{G^{-1}(\tau)}{\tau^{1+\frac{s}{n}}}\,d\tau\right)+G\left(\int_1^{\frac{1}{t}}\frac{G^{-1}(\tau)}{\tau^{1+\frac{s}{n}}}\,d\tau\right)\right)  :=(a)+(b).
\end{align*}
As before, $(a)\leq \frac{\C}{2} C_1 t$. 

In order to bound $(b)$, observe that
\begin{equation} \label{crece}
G^{-1}(\tau)\tau^{-\frac{s}{n}} \text{ is an increasing function for all }\tau >0
\end{equation}
whenever $sp^+<n$. Indeed, by using the change of variable $w=G(\tau)$, the last assertion is equivalent to $w G(w)^{-\frac{s}{n}}$ to be increasing; this, in turn, is impllied by \eqref{cond} and the assumption that that $p^+<\frac{n}{s}$ as we can compute
$$
(w G(w)^{-\frac{s}{n}})' = G(w)^{-\frac{s}{n}}-\frac{sw}{n}  G(w)^{-\frac{s}{n}-1}g(w)   
$$
and this expression is positive as $$\frac{wg(w)}{G(w)}\leq p^+<\frac{n}{s}.$$

Therefore, from \eqref{crece} we have that
\begin{align*}
tG\left(\int_1^{\frac{1}{t}}\frac{G^{-1}(\tau)}{\tau^{1+\frac{s}{n}}}\,d\tau\right) 
&\leq 
tG\left( G^{-1}\left(\frac1t\right)t^{\frac{s}{n}}\int_1^{\frac{1}{t}} \tau^{-1 }\,d\tau\right)\\
&\leq
tG\left( G^{-1}\left(\frac1t\right)t^{\frac{s(1-\ve)}{n}} t^{\frac{s\ve}{n}} \log\left(\frac1t\right) \right)
\end{align*}
where $\ve>0$ is arbitrary. Since $t^{\frac{s\ve}{n}} \log\left(\frac1t\right) \leq C_\ve$, from the last expression together with  \eqref{minmax2} and the fact that $t<1$ we get
\begin{align*}
tG\left(\int_1^{\frac{1}{t}}\frac{G^{-1}(\tau)}{\tau^{1+\frac{s}{n}}}\,d\tau\right) 
&\leq \tilde C_\ve
t^{1+\frac{sp}{n}(1-\ve)p^- }G\left( G^{-1}\left(\frac1t\right)  \right) \leq \tilde C_\ve t^\frac{sq}{n}
\end{align*}
for any $q<p^-$.

Putting together the bounds for $(a)$ and $(b)$ we get the desired estimate.
\end{proof}

Next we get a bound for the $H$ norm of $G(u)$ in terms of the $G^*$ norm:	

\begin{lema} \label{lema.norma}
Let $G$ be a Young function satisfying \eqref{cond}. Then for any $u\in L^G(\Omega)$ we have that
$$
\|G(u)\|_H \leq  \max\{ \|u\|_{G^*}^{p^+}, \|u\|_{G^*}^{p^-}\}.
$$
\end{lema}
\begin{proof}
When $\|u\|_{G^*}\geq 1$,   \eqref{minmax3} and the definition of the Luxemburg norm yield that
\begin{align*}
\int_\Omega H\left( \frac{G(u)}{\|u\|_{G^*}^{p^+}} \right)\,dx = \int_\Omega G^*\circ G^{-1}\left( \frac{G(u)}{\|u\|_{G^*}^{p^+}} \right)\,dx \leq \int_\Omega G^* \left( \frac{u}{\|u\|_{G^*}} \right)\,dx =1
\end{align*}
which gives $\|G(u)\|_H\leq \|u\|_{G^*}^{p^+}$. The case $\|u\|_{G^*}<1$ is analogous.
\end{proof}

The following lemma relates norms with modulars of weak solutions to \eqref{eq.main}. This is crucial as the H\"older inequality is given in terms of norms, whereas the notion of weak solutions is related to modulars.

\begin{lema}\label{ecuacion}
If $u\in W^{s,G}(\Omega)$  satisfies that
$$
\iint_{\R^{n}\times\R^n} G(|D_s u|)\,d\mu \leq  M
$$
for some $M\geq 1$, then it holds that
$$
[u]_{s,G} \leq  M^\frac{1}{p^-}.
$$
\end{lema}

\begin{proof}
Since $M \geq 1$, from \eqref{minmax2} we have
\begin{align*}
1&\geq \frac{1}{M} \iint_{\R^{n}\times\R^n} G(|D_s u|)\,d\mu 
\geq 
\iint_{\R^{n}\times\R^n}  G\left(M^{-\frac{1}{p^-}} |D_s u| \right)\,d\mu.
\end{align*}
 Then,  the definition of $\left[ \cdot\right]_{s,G}$ implies  that $\left[ u \right]_{s,G} \leq  M^\frac{1}{p^-}$ as desired.
\end{proof}

We also have the general version of the Chebyshev's inequality:

\begin{lema} \label{chebyshev} %https://sites.math.washington.edu/~morrow/336_17/papers17/george.pdf teo 3.1
Let $G$ be a real valued, measurable in $\Omega$, nonnegative and nondecreasing function. For any $u$ measurable in $\Omega$ and real valued and $t>0$ we have
$$
|\{x\in\Omega\colon u(x)\geq t \}| \leq \frac{1}{G(t)}\int_\Omega G(u(x))\,dx.
$$
\end{lema}

\begin{proof}
The proof is the usual one:
\[
G(t)|\{x\in\Omega\colon u(x)\geq t \}|=\int_{\{x\in\Omega\colon u(x)\geq t \}} G(t)\:dx \leq \int_{\{x\in\Omega\colon u(x)\geq t \}}G(u(x))\,dx\leq \int_\Omega G(u(x))\,dx.
\]
\end{proof}

We close this section with a simple real analysis result regarding sequences that satisfy a nonlinear recurrence relationship:
\begin{lema}\label{sequence}
Let $\{a_k\}_k$ be a sequence of nonnegative real numbers and assume that exist $\bar{C},\tilde{C}>0$ and $\delta\in(0,1)$ such that
\begin{equation}\label{recur}
a_{k+1}\leq \bar{C}\tilde{C}^{k+1}a_k^{1+\delta}, \quad k\geq0.
\end{equation}

Then there exists $\varepsilon_0>0$ such that 
\[
a_0\leq \varepsilon_0\quad\Rightarrow\quad\lim_{k\rightarrow\infty}a_k=0.
\] 
\end{lema}

\begin{proof}
We start by iterating \eqref{recur}:
\begin{align*}
& a_1\leq \bar{C}\tilde{C}a_0^{1+\delta} \\
& a_2\leq \bar{C}\tilde{C}^2a_1^{1+\delta} \leq \bar{C}^{2+\delta}\tilde{C}^{3+\delta}a_0^{2(1+\delta)}\\
& \vdots \\
& a_k\leq \bar{C}^{k+\delta}\tilde{C}^{k+1+\delta}a_0^{k(1+\delta)}=\bar{C}^\delta\tilde{C}^{1+\delta}\left(\bar{C}Ca_0^{1+\delta}\right)^k.
\end{align*}
Therefore, if $a_0\leq \varepsilon_0<\left(\frac{1}{\bar{C}\tilde{C}}\right)^{\frac{1}{1+\delta}}$ we get
$\lim_{k\rightarrow\infty}a_k=0$ as desired.
\end{proof}

\section{Proof of the main results}\label{sec.main}

We start this section with the proof of our main result, namely Theorem \ref{main}. The strategy is to use a De Giorgi-type argument to get that soltutions with small enough modular are bounded.

\begin{proof}[Proof of Theorem \ref{main}]
The proof follows De Giorgi's $L^2$ implies $L^\infty$ scheme; we are going to show that if
\begin{equation}\label{equation}
\iint G(D_s u)\,d\mu \leq c \lambda \int_\Omega G(u)\,dx
\end{equation}
for some constant $c=c(p^+,p^-)$, then there exists $\varepsilon_0>0$ such that 
\begin{equation}\label{l2tolinfty}
\int_\Omega G(u)=\mu\leq \varepsilon_0\quad\Rightarrow \quad \|u\|_{L^\infty(\Omega)} \leq 1.
\end{equation}
Note that \eqref{equation} is readily implied by \eqref{eq.weak} (taking $v=u$) and \eqref{cond}:
\begin{align*}
\iint G(D_s u)\,d\mu & \leq \frac{1}{p^-}\iint g(D_s u)D_s u\,d\mu  \\
						  & = \frac{1}{p^-}	\lam \int_\Omega g(u)u\,dx \\
						  & \leq \frac{p^+}{p^-}	\lam \int_\Omega G(u)\,dx.  
\end{align*}

Also, \eqref{l2tolinfty} implies the general result by scaling: if $\mu>\varepsilon_0$ we can rescale:

%\textcolor{blue}{
%ac\'a la constante $C$ se puede elegir independiente de $\|u\|_\infty$, y dependiendo solo de $\mu$ y $\ve_0$ (aunque $\ve_0$ de cierto modo depende de $u$), porque sino llegar\'iamos a que $\|u\|_\infty\leq C$ pero $C$ puede ser de nuevo $\|u\|_\infty$, no?
%}
%\textcolor{red}{S\'i, ac\'a $C$ depende \'unicamente de $\mu$ y $\varepsilon_0$}
%
\[
\int_\Omega G\left(\frac{u}{C}\right)\:dx\leq \max\left\{C^{-p^-},C^{-p^+}\right\}\mu\leq \varepsilon_0
\]
by taking $C$ sufficiently large depending only on $\mu$ and $\varepsilon_0$ . Notice that $u/C$ fulfills that
\begin{align*}
\iint_{\R^{n}\times\R^n} G\left(\left|D_s \left(\frac{u}{C}\right)\right|\right)\,d\mu & \leq  \max\left\{C^{-p^-},C^{-p^+}\right\}\iint_{\R^{n}\times\R^n} G(|D_su|)\,d\mu \\
																		  &	\leq  \max\left\{C^{-p^-},C^{-p^+}\right\} c \lam \int_\Omega G(u)\,dx \\
																		  &	\leq  \frac{\max\left\{C^{-p^-},C^{-p^+}\right\}}{\min\left\{C^{-p^-},C^{-p^+}\right\}}c\lam \int_\Omega G\left(\frac{u}{C}\right)\,dx \\
																		  &	=: C_0\lam \int_\Omega G\left(\frac{u}{C}\right)\,dx.
\end{align*}
Then, by \eqref{l2tolinfty} we have
\[
\left\|u\right\|_{L^\infty(\Omega)}\leq C
\]
and we get the desired result.

Let us prove \eqref{l2tolinfty}. For any $k\in\N$ consider the function $w_k\in W^{s,G}_0(\Omega)$ defined as
$$
w_k:=(u-(1-2^{-k}))_+.
$$
It is easy to see that these functions fulfill the following properties
\begin{align} \label{des0}
\begin{split}
&w_{k+1}(x) \leq w_k(x) \quad\text{a.e. in }\R^n,\\
&\{w_{k+1}>0\}\subset \{w_k> 2^{-(k+1)}\}.
\end{split}
\end{align}
We further claim that:
\begin{equation} \label{des00}
u\leq (2^{k+1}-1) w_k \quad\text{in } \{w_{k+1}>0\}.
\end{equation}
Indeed, notice that $w_{k+1}(x)>0$ implies $u(x)>1-2^{-(k+1)}$ and that 
\[
2^{k+1}-1=\frac{1-2^{-(k+1)}}{1-2^{-(k+1)}-(1-2^{-k})}
\]
and compute
\begin{align*}
(2^{k+1}-1)w_k(x) & = (2^{k+1}-1)\left(u(x)-(1-2^{-k})\right) \\
				  & = \frac{1-2^{-(k+1)}}{1-2^{-(k+1)}-(1-2^{-k})}u(x)-\frac{(1-2^{-(k+1)})(1-2^{-k})}{1-2^{-(k+1)}-(1-2^{-k})}\\
				  & = u(x)+\frac{1-2^{-k}}{1-2^{-(k+1)}-(1-2^{-k})}u(x)-\frac{(1-2^{-(k+1)})(1-2^{-k})}{1-2^{-(k+1)}-(1-2^{-k})}\\
				  & = u(x)+2^{k+1}(1-2^{-k})\left(u(x)-(1-2^{-(k+1)})\right) >u(x),
\end{align*} 
so \eqref{des00} holds.

Now, since $0\leq w_k \leq |u| + 1 \in L^G(\Omega)$ and 
\[
\lim_{k\to\infty} w_k = (u-1)_+,
\]  
by the Dominated Convergence Theorem one gets that
\begin{equation}\label{eq.dominated}
\lim_{k\to\infty} \int_\Omega G(w_k)\,dx = \int_\Omega G((u-1)_+)\,dx.
\end{equation}

We want to get a recursive bound of the form
\begin{equation}\label{eq.bound}
\int_{\Omega} G(w_{k+1})\,dx \leq C_k\left(\int_{\Omega} G(w_k)\,dx\right)^{1+\delta}
\end{equation}
for some $\delta>0$ and some (increasing) sequence of constants $C_k>0$. Indeed, \eqref{eq.bound} is exactly condition \eqref{recur} in Lemma \ref{sequence} so its proof would imply
\[
\int_{\Omega} G(u_+)\,dx=\int_{\Omega} G(w_0)\,dx \leq \varepsilon_0\quad \Rightarrow\quad \lim_{k\rightarrow\infty}\int_\Omega G(w_k)\,dx =0.
\]
Finally this combined with \eqref{eq.dominated}, implies 
\[
u\leq 1\quad\text{ a.e. in } \Omega.
\]
Replacing $u$ by $-u$ we get the other bound.

To prove \eqref{eq.bound} we start with the following inequality:
\begin{equation}\label{desig}
g\left(\frac{v(x)-v(y)}{|x-y|^s}\right)\left(\frac{v_+(x)-v_+(y)}{|x-y|^s}\right)\geq p^-G\left(\frac{v_+(x)-v_+(y)}{|x-y|^s}\right).
\end{equation}
Indeed, we may assume without loss of generality that $v(x)\geq v(y)$. If $x,y\in\{v>0\}$ then \eqref{desig} is just \eqref{cond}. If $x\in\{v>0\}$ and $y\in\{v\leq 0\}$ then we use the fact that $g$ is increasing and \eqref{cond} to get
\[
g\left(\frac{v(x)-v(y)}{|x-y|^s}\right)\left(\frac{v_+(x)-v_+(y)}{|x-y|^s}\right)\geq g\left(\frac{v(x)}{|x-y|^s}\right)\left(\frac{v(x)}{|x-y|^s}\right) \geq p^-G\left(\frac{v_+(x)-v_+(y)}{|x-y|^s}\right)
\]
as desired.

Now we use \eqref{desig} with $v=w_{k+1}$ as follows:
\begin{align*}
\iint_{\R^{2n}} G(D_s w_{k+1})\,d\mu & = \iint_{\R^{2n}} G\left(\frac{w_{k+1}(x)-w_{k+1}(y)}{|x-y|^s}\right)\frac{dxdy}{|x-y|^n}  \\
									 & \leq \frac{1}{p^-}\iint_{\R^{2n}}g\left(\frac{u(x)-u(y)}{|x-y|^s}\right)\left(\frac{w_{k+1}(x)-w_{k+1}(y)}{|x-y|^s}\right)\frac{dxdy}{|x-y|^n}   \\ 
%									 & \leq \frac{1}{p^-}\iint_{\R^{2n}}g\left(D_su\right)\left(D_s w_{k+1}\right)\frac{dxdy}{|x-y|^{n}}   \\ 
									 & = \frac{\lam}{p^-} \int_{\Omega} g(u)w_{k+1} \,dx 
\end{align*}
where the last equality comes from testing the equation with $w_{k+1}$. Using this together with \eqref{des00}, \eqref{des0} and \eqref{cond} gives
\begin{align} \label{subsol}
\begin{split}
\iint_{\R^{2n}} G(D_sw_{k+1})\,d\mu &\leq  
\frac{\lam}{p^-} \int_{\Omega} g( (2^{k+1}-1)w_{k+1} ) \frac{(2^{k+1}-1)w_{k+1}}{2^{k+1}-1} \,dx \\
&\leq 
\frac{p^+}{p^-}\lam \int_{\Omega} G( (2^{k+1}-1)w_{k+1} ) \frac{1}{2^{k+1}-1} \,dx\\
&\leq 
\frac{p^+}{p^-} \lam (2^{k+1}-1)^{p^+-1}\int_{\Omega} G(w_{k+1}) \,dx.
\end{split}
\end{align}
Next, by using H\"older inequality for Orlicz spaces
\begin{equation}\label{eq.hold}
\int_{\Omega} G(w_{k+1})\,dx \leq2 \| G(w_{k+1})\|_H \| \chi_{\{w_{k+1}>0\}}\|_{\tilde H}  
\end{equation}
where $H=G^*\circ G$ and $\tilde{H}$ is its conjugate. To get a bound for the first factor we recall that 
\begin{align*}
\| \chi_{\{w_{k+1}>0\}}\|_{\tilde H}  \leq |\{w_{k+1}>0\}| H^{-1}\left(|\{w_{k+1}>0\}|^{-1}\right) = :K\left(|\{w_{k+1 >0}\}|\right)
\end{align*}
(see \cite{kufner1977function}, page 149) to get, using Lemma \ref{cotaK}, \eqref{des0},  Lemma \ref{chebyshev} and \eqref{minmax2}
\begin{align}\label{cota1} 
\begin{split}
\|\chi_{\{w_{k+1}>0\}}\|_{\tilde H}  &\leq   \bar{C} \kappa (|\{w_{k+1}>0\}|)\\
									 &\leq   \bar{C} \kappa (|\{w_k>2^{-(k+1)}\}|)\\
									 &\leq \bar{C} \kappa \left( \frac{1}{G( 2^{-(k+1)})} \int_{\Omega} G(w_{k})\,dx \right)\\
									 &\leq \bar{C}\kappa \left( \frac{1}{G(1) 2^{-(k+1)p^+}} \int_{\Omega} G(w_{k})\,dx \right)\\
									 &\leq \bar{C}\kappa \left(2^{p^+}\right)^{k+1}\kappa\left( \int_{\Omega} G(w_{k})\,dx \right)\\
									 &=:\bar{C}\tilde{C}^{k+1}\kappa\left( \int_{\Omega} G(w_{k})\,dx \right)\\
\end{split}
\end{align}
where $\kappa$ denotes the increasing function $\kappa(t)=\max\{t,t^\frac{sq}{n}\},\:q<p^-$, $t>0$. Recall that we normalized so that $G(1)=1$ without loss of generality.

%Here $p^\pm$ stands for either $p^+$ o $p^-$ depending on whether $|\{w_{k+1}>0\}|$ is bigger or smaller than 1 (note that this may vary but it stabilizes for $k$ large enough).

Now we need to bound the other term in \eqref{eq.hold}. For this, we use Lemma \ref{lema.norma}, Proposition \ref{embedding}, and Lemma \ref{ecuacion} applied to \eqref{subsol}
\begin{align*}
 \| G(w_{k+1})\|_H &\leq \max\{ \|w_{k+1}\|_{G^*}^{p^+},\|w_{k+1}\|_{G^*}^{p^-} \} \\
 &\leq C \max\{ [w_{k+1}]_{s,G}^{p^+}, [w_{k+1}]_{s,G}^{p^-} \}.
\end{align*}
%\textcolor{red}{el punto clave es este, que con esta elecci\'on de M para usar el lema \ref{lema.norma}, siempre podemos asumir que $M\geq 1$)} 
Applying Lemma \ref{ecuacion} to \eqref{subsol} with $M=\frac{p^+}{p^-} \lam (2^{k+1}-1)^{p^+-1}\int_{\Omega} G(w_{k+1}) \,dx$ gives 
$$
[w_{k+1}]_{s,G} \leq C(\lambda,p^\pm) 2^{\frac{(k+1)(p^+-1)}{p^-}} \left(\int_\Omega G(w_k) \right)^\frac{1}{p^-}.
$$
Observe that $C(\lambda,p^\pm) 2^{\frac{(k+1)(p^+-1)}{p^-}} >1$ for $k$ big enough.

The last two inequalities together give 
\begin{equation} \label{cota2}
 \| G(w_{k+1})\|_H\leq  \tilde{C}^{k}  \max\left\{ \int_\Omega G(w_k)\,dx, \left(\int_\Omega G(w_k)\,dx \right)^\frac{p^+}{p^-} \right\}
\end{equation}
where we can take $\tilde{C}$ the same as in \eqref{cota1} by making it bigger if necessary. 

Inserting \eqref{cota1} and \eqref{cota2} in \eqref{eq.hold} (and relabeling $\tilde{C}$ once more) we finally get
\[
\int_{\Omega} G(w_{k+1})\,dx \leq  \bar{C}\tilde{C}^{k+1} \kappa\left( \int_{\Omega} G(w_{k})\,dx \right)\max \left\{ \int_{\Omega}G(w_{k}) ,  \left( \int_{\Omega}G(w_{k}) \right)^{\frac{p^+}{p^-}} \right\}.
\]
We conclude noticing that both if $\int_{\Omega}G(w_{k})\leq1$ or $\int_{\Omega}G(w_{k})>1$ the previous inequality leads to 
\[
\int_{\Omega} G(w_{k+1})\,dx \leq \bar{C}\tilde{C}^{k+1} \left(  \int_{\Omega}G(w_{k}) \right)^{1+\delta}
\]
for some (possibly different) $\delta>0$\footnote{We point out that the sequence $\left\{\int_{\Omega} G(w_{k})\,dx\right\}_k$ is monotone and hence the possibility of different $\delta$s is not an issue.}. This proves \eqref{eq.bound} and hence the theorem. 

\end{proof}

Once Theorem \ref{main} is proven, it is rather immediate to get the

\begin{proof}[Proof of Corollary \ref{cor}]
Simply define 
\[
h(x):=\lambda g(u(x))
\]
which by Theorem \ref{main} belongs to $L^\infty(\Omega)$ and apply Theorem \ref{thm.reg}.
\end{proof}

Finally, the proof of Theorem \ref{thm.semilinear} is the same as that of Theorem \ref{main} and Corollary \ref{cor} mutatis mutandis.

\subsection*{Acknowledgements.} This work was partially supported by ANPCyT under grant PICT 2019-3530. All three authors are members of CONICET.

%%%%%%%%%%%%%%%%%%%%%%%%%%%%%%%%%%%%%%%%%%%%%%%%

%%%%%%%%%%%%%%%%%%%%%%%%%%%%%%%%%%%%%%%%%%%%%%%%

\end{document}